\documentclass[11pt,a4paper,twoside]{amsart}
\usepackage{amsmath,amsfonts,amsthm,amsopn,color,amssymb,enumitem}
\usepackage{palatino}
\usepackage{graphicx}
\usepackage[colorlinks=true]{hyperref}
\hypersetup{urlcolor=blue,citecolor=red,linkcolor=blue}

\usepackage{relsize}
\usepackage{esint}

\newcommand\e\varepsilon

\newcommand\R{\mathbb R}

\newcommand\de\partial
\newcommand\weakto\rightharpoonup

\renewcommand\le\leqslant
\renewcommand\ge\geqslant
\renewcommand\a\alpha

\renewcommand\b\beta

\renewcommand\d\delta
\newcommand\vfi\varphi
\newcommand\g\gamma
\newcommand\gb\gamma
\renewcommand\l\lambda
\newcommand\n\nabla
\newcommand\s\sigma
\renewcommand\t\theta
\renewcommand\O\S

\newcommand\G\Gamma

\renewcommand\S\Sigma
\renewcommand\L\Lambda

\newcommand\N{\mathbb N}

\renewcommand\o\S

\def\bbm[#1]{\text{\boldmath $#1$}}
\newcommand\beq{\begin{equation}}
\newcommand\eeq{\end{equation}}

\newtheorem{theorem}{Theorem}[section]
\newtheorem{lemma}[theorem]{Lemma}

\newtheorem{proposition}[theorem]{Proposition}
\newtheorem{remark}[theorem]{Remark}
\newtheorem{corollary}[theorem]{Corollary}

\numberwithin{equation}{section}

\def\sideremark#1{\ifvmode\leavevmode\fi\vadjust{\vbox to0pt{\vss
\hbox to0pt{\hskip\hsize\hskip1em
\vbox{\hsize3cm\tiny\raggedright\pretolerance10000
\noindent #1\hfill}\hss}\vbox to8pt{\vfil}\vss}}}

\title[A mean field problem approach for the double curvature prescription problem]
{A mean field problem approach for the double curvature prescription problem}

\author{Luca Battaglia}
\address{Luca Battaglia\\ 
Universit\`a degli Studi Roma Tre\\
Dipartimento di Matematica e Fisica\\
Via della Vasca Navale 84\\
00146 Roma, Italy.}
\email{luca.battaglia@uniroma3.it}

\author{Rafael L\'opez--Soriano}
\address{Rafael L\'opez--Soriano\\
Universidad de Granada\\
IMAG\\
Departamento de Análisis Matem\'atico\\
Campus Fuentenueva\\
18071 Granada, Spain.}
\email{ralopezs@ugr.es}

\keywords{Prescribed curvature problem, conformal metric, variational methods.\\The authors certify they have NO conflict of interest.}

\subjclass[2010]{35J20, 58J32.}

\everymath{\displaystyle}
\allowdisplaybreaks

\begin{document}

\begin{abstract}
In this paper we establish a new mean field-type formulation to study the problem of prescribing Gaussian and geodesic curvatures on compact surfaces with boundary, which is equivalent to the following Liouville-type PDE with nonlinear Neumann conditions:
$$\left\{\begin{array}{ll}
-\Delta u+2K_g=2Ke^u&\text{in }\S\\
\de_\nu u+2h_g=2he^\frac u2&\text{on }\de\S.
\end{array}\right.$$
The underlying problem allows the application of straightforward variational techniques. Consequently, we provide three different existence results in the cases of positive, zero and negative Euler characteristics by means of variational techniques.
\end{abstract}

\maketitle

\section{Introduction}

Let $\S$ be a compact surface with boundary equipped with a metric $g$. We consider the following Neumann boundary value problem

\beq\label{gg0}
\left\{\begin{array}{ll}
-\Delta u+2K_g=2Ke^u&\text{in }\S\\
\de_\nu u+2h_g=2he^\frac u2&\text{on }\de\S,
\end{array}\right.
\eeq
where $\Delta=\Delta_g$ is the Laplace--Beltrami operator in $(\S,g)$, $\de_\nu$ is the normal derivative with $\nu$ being the outward normal vector to $\de\S$ and $K_g,K:\S\to\R$ and $h_g,h:\de\S\to\R$ are smooth.\\
This equation has a special interest due to its geometric meaning. In fact, it is equivalent to prescribing at the same time the Gaussian curvature in $\S$ and the geodesic curvature on $\de\S$. More precisely, given a metric $\tilde g=ge^u$ conformal to $g$, if $K_g,K$ are the Gaussian curvatures and $h_g,h$ are the geodesic curvatures of $\de\S$ with respect to the metrics $g,\tilde g$, then $u$ satisfies \eqref{gg0}.\\

Problem \eqref{gg0} can be seen as a natural generalization of the following Liouville-type equation on closed surfaces
\beq\label{liouville}
-\Delta u+2K_g=2Ke^u\qquad\text{in }\S,
\eeq
which is equivalent to prescribing the Gaussian curvature $K$ on $\S$.\\
Problem \eqref{liouville} has been very widely studied for several decades. For a survey on this topic we refer to the Chapter 6 in \cite{aubin} and the references therein.\\
On the other hand, literature concerning problem \eqref{gg0} is not as wide, unless in some special cases. For instance, the case of zero prescribed geodesic curvature $h\equiv0$ has been treated in \cite{chang2,liliu,liuhuang}, whereas the case of zero prescribed Gaussian curvature $K\equiv0$ in \cite{chyg}. The special case of constant curvatures $K\equiv K_0,h\equiv h_0$ has been studied in \cite{brendle}. In that case, solutions have been explicitely classified when $\S$ is a disk (see \cite{lizhu,zhang,mira-galvez}) or an annulus (\cite{asun}).\\
In the case of non-constant curvatures, problem \eqref{gg0} has been studied in \cite{lsmr} for a wide range of situations, under the assumption of negative Gaussian curvature $K(x)<0$ and non-positive Euler characteristic $\chi(\S)\le0$. The authors provided existence, uniqueness and non-existence results as well as a blow-up analysis.\\

In order to study problem \eqref{gg0}, some information can be easily deduced from the Gauss-Bonnet theorem. In fact, integrating both sides of \eqref{gg0} and then by parts, one obtains:
\beq\label{gb}
\int_\S Ke^u+\int_{\de\S}he^\frac u2=\int_\S K_g+\int_{\de\S}h_g=2\pi\chi(\S),
\eeq
which imposes necessary conditions on the functions $K,h$.\\
In particular, if $\chi(\S)>0$, that is $\S$ is topologically equivalent to the unit disk, then $K$ or $h$ must be positive somewhere in $\S$ or $\de\S$, respectively. Similarly, if $\chi(\S)<0$, then $K$ or $h$ must be negative somewhere. On the other hand, if $\chi(\S)=0$ then either $K$ or $h$ must change sign, or one of the two is everywhere non-negative and the other is everywhere non-positive.\\

It can be seen, using the Uniformization Theorem, that one can always prescribe zero geodesic curvature and constant Gaussian curvature equal to $\mathrm{sign}(\chi(\S))$ (more details can be found in \cite{lsmr}, Proposition 3.1). Geometrically, this is equivalent to working on a half-sphere in the case $\chi(\S)>0$, on a cylinder in the case $\chi(\S)=0$ and on a domain of the hyperbolic plane in the case $\chi(\S)<0$.\\
Therefore, one can assume without losing generality that $K_g\equiv\frac{2\pi\chi(\S)}{|\S|},h_g\equiv0$ and consider the following problem:
\beq\label{doublecurv}
\left\{\begin{array}{ll}
-\Delta u+\frac{4\pi\chi(\S)}{|\S|}=2Ke^u&\text{in }\S\\
\de_\nu u=2he^\frac u2&\text{on }\de\S.
\end{array}\right.
\eeq
From now on, we will focus on problem \eqref{doublecurv} instead of \eqref{gg0} without any further comment.\\

In \cite{lsmr,lssrr} solutions to \eqref{doublecurv} are found as critical points of the energy functional
$$\mathcal I(u):=\int_\S\left(\frac{|\nabla u|^2}2+\frac{4\pi\chi(\S)}{|\S|}u-2Ke^u\right)-4\int_{\de\S}he^\frac u2.$$
Such a formulation turns out to be very convenient when both $K$ and $\chi(\S)$ are non-positive, but it is not clear how it can be extended to other cases.\\
In this paper we propose an alternative and equivalent formulation to attack problem \eqref{doublecurv}, which seems to be more useful in several cases.\\
In particular, we will write problem \eqref{doublecurv} in an equivalent mean field form, which simplifies the variational analysis in a way. This approach is a widely-used tool to attack problem \eqref{liouville} on closed surfaces and have also been used in \cite{chyg,chang2,liliu,liuhuang} for problem \eqref{doublecurv} when either $K$ or $h$ identically vanishes.\\
Another mean field formulation for problem \eqref{doublecurv} was given in \cite{sergio}, although it seems to be working only if both curvature $K,h$ and $\chi(\S)$ are all positive. A double mean field problem related to \eqref{doublecurv} was also studied by the authors in \cite{bls}. However, such results of existence do not apply for the geometrical problem.\\
On the other hand, the mean field formulation here presented is totally equivalent in most of the cases and in all cases produces solutions for problem \eqref{doublecurv}. This will be discussed in further detail in Section 2.\\

The new mean field energy functional is defined on an open subset of $H^1(\S)$, which in general does not coincide with the whole space, and it has different forms depending whether $\chi(\S)$ is positive, zero or negative. The key role played by the sign of the Euler characteristic is consistent with the existence results from \cite{lsmr} as well as for many results concerning problem \eqref{liouville}.\\
In the case of a disk the new energy functional shares many similarities with Liouville-type problem where only one curvature is prescribed. In particular, using Trudinger-Moser type inequalities, one can show that the energy is bounded from below though not coercive. Section 3 elaborates further on this aspect.\\
Quite surprisingly, the picture is rather different for surfaces of non-positive genus. In fact, the main nonlinear term is now related to the interaction between the two nonlinear terms, which can have different signs.\\
In particular, in the case of negative $K(x)$, an important role will be played by the quotient function $\mathfrak D:\de\S\to\R$, defined by:
\beq\label{d}
\mathfrak D(x):=\frac{h(x)}{\sqrt{|K(x)|}}.
\eeq
The importance of $\mathfrak D(x)$ has already emphasized in \cite{lsmr,lssrr,bcp}.\\

We are now in position to state the main theorems of this work.\\
We start with the case $\chi(\S)>0$. Here, we need to prescribe a positive curvature $K$ in order to have coercivity of the energy functional on its domain of definition, which may have a nonempty boundary where the energy does not diverge; this is a new difficulty arising from this formulation. However, this does not seem to be a mere technicality, since when $K$ is not positive the problem may have no solutions, even in the case of constant prescribed curvatures on the disk (see \cite{zhang}).\\
On the other hand, to get coercivity we also need to assume some symmetry on $K,h$. In fact, this allows to look for solutions in the space of functions with the same symmetry of the prescribed curvatures, where one gets an improved constant in the Trudinger-Moser inequality. The role of symmetry was first highlighted in \cite{moser} for problem \eqref{liouville} and then in \cite{sergio} for problem \eqref{doublecurv}.\\
To be more precise, we will assume $\S$ to be the upper half-sphere
\beq\label{s}
\S:=\left\{(x_1,x_2,x_3)\in\R^3:\,x_1^2+x_2^2+x_3^2=1,x_3>0\right\},
\eeq
which we just showed to be non-restrictive, and we assume $K,h$ to be symmetric with respect to a subgroup of the orthonormal group of the plane $SO(2)$ having no fixed points on the unit circle $\mathbb S^1=\de\S$. This is equivalent to assume $K,h$ to be $k$-symmetric for some positive integer $k\in\N$ with $k\ge2$ (which one may also assume to be prime, without losing generality), namely:
$$K(\rho_kx)=K(x)\quad\forall\,x\in\S,\qquad h(\rho_ky)=h(y)\quad\forall\,y\in\de\S,$$
where
\beq\label{rhok}
\rho_k(x_1,x_2,x_3)=\left(x_1\cos\frac{2\pi}k-x_2\sin\frac{2\pi}k,x_1\sin\frac{2\pi}k+x_2\cos\frac{2\pi}k,x_3\right).
\eeq

\begin{theorem}\label{thm:chipositive}
Assume $\S$ is given by \eqref{s} and:
\begin{itemize}
\item $K(x)\ge0$ for all $x\in\S$, $K\not\equiv0$ and $K,h$ are $k$-symmetric.
\end{itemize}
Then problem \eqref{doublecurv} admits a solution.
\end{theorem}\

Let us point out that with respect to the results given in \cite{sergio}, the previous theorem does not include any obstruction regarding the sign of the curvature $h$. Due to a perturbation argument used in \cite{sergio}, we can extend the preceding result for sign--changing $K$ with a small negative part. See Theorem~\ref{thm:perturb} for a complete statement of this result.

In the case $\chi(\S)=0$ we have two possible scenarios. If $K$ is somewhere positive, then the nonlinear term is everywhere negative, hence the energy functional is coercive and the problem admits a minimizing solution.\\
On the other hand, if $K$ is always negative, then we require a condition on the quantity $\mathfrak D$ to deal with the nonlinear term and get minimizers. This is actually a particular case of Theorem 1.2 in \cite{lsmr}. In both scenarios we shall make some assumptions on the sign of $h$ in order to avoid loss of coercivity in the boundary of the domain of the energy functional, as in Theorem \ref{thm:chipositive}.

\begin{theorem}\label{thm:chizero}
Assume $\chi(\S)=0$ and one of the following occurs:
\begin{enumerate}
\item $K(x)>0$ for some $x\in\S$ and $h(y)\le0$ for all $y\in\de\S$, $h\not\equiv0$;
\item $K(x)\ge0$ for all $x\in\S$, $K\not\equiv0$ and $\int_{\de\S}h<0$;
\item $K(x)\le0$ for all $x\in\S$, $K\not\equiv0$ and $h\not\equiv0,h(y)>0,\mathfrak D(y)<1$ for all $y\in\de\S$.
\end{enumerate}
Then problem \eqref{doublecurv} admits a solution.
\end{theorem}\

Finally, in the case $\chi(\S)<0$ we get minimizing solutions under an assumption on $\mathfrak D$, similarly as Theorem \ref{thm:chizero}. This is the same result as Theorem 1.1 in \cite{lsmr}, although obtained with different arguments.

\begin{theorem}\label{thm:chinegative}
Assume $\chi(\S)<0$ and:
\begin{itemize}
\item $K(x)\le0$ for all $x\in\S$, $K\not\equiv0$ and $\mathfrak D(y)<1$ for all $y\in\de\S$.
\end{itemize}
Then problem \eqref{doublecurv} admits a solution.
\end{theorem}

We point out that Theorem \ref{thm:chinegative} includes the case of a non-positive $h(y)\le0$, which can be equivalently written as $\mathfrak D(y)\le0$ for all $y\in\de\S$.\\

The paper is structured as follows. In Section 2 we introduce the new equivalent formulation for problem, \eqref{doublecurv}, the underlying energy functional and an appropriate functional setting. In Section 3 we assume $\chi(\S)>0$ and we prove Theorem \ref{thm:chipositive}. In Section 4 we assume $\chi(\S)\le0$, which requires different arguments with respect to the previous case, and we provide the proofs for Theorems \ref{thm:chizero} and \ref{thm:chinegative}.

\

\section{Mean field formulation and setting}

We start by introducing a new mean field problem and showing its equivalence with problem \eqref{doublecurv}.\\
Actually, the mean field problem will be a stronger formulation, in the sense that its solvability will imply the solvability of problem \eqref{doublecurv}, but in most of the cases we will consider the two problems are in fact equivalent.

\begin{proposition}\label{meanfield}
Problem \eqref{doublecurv} has a solution if the following problem has a solution:
\beq\label{meanfieldeq}
\left\{\begin{array}{ll}
-\Delta u+\frac{4\pi\chi(\S)}{|\S|}=2C(u)^2Ke^u&\text{in }\S\\
\de_\nu u=2C(u)he^\frac u2&\text{on }\de\S,\\
C(u)>0
\end{array}\right.
\eeq
where $C(u)$ is a constant defined by:
\beq\label{c}
C(u):=\left\{\begin{array}{ll}\frac{4\pi\chi(\S)}{\sqrt{\left(\int_{\de\S}he^\frac u2\right)^2+8\pi\chi(\S)\int_\S Ke^u}+\int_{\de\S}he^\frac u2}&\text{if }\chi(\S)>0\\-\frac{\int_{\de\S}he^\frac u2}{\int_\S Ke^u}&\text{if }\chi(\S)=0\\\frac{4\pi|\chi(\S)|}{\sqrt{\left(\int_{\de\S}he^\frac u2\right)^2-8\pi|\chi(\S)|\int_\S Ke^u}-\int_{\de\S}he^\frac u2}&\text{if }\chi(\S)<0.\end{array}\right.
\eeq
Moreover, if either $\chi(\S)=0$ or $\chi(\S)K(x)>0$ for all $x\in\S$, then the two problems are equivalent.
\end{proposition}

\begin{proof}

Let $u$ be a solution of \eqref{meanfieldeq}. Then for any $c\in\R$, the function $v=u+c$ satisfies:
$$\left\{\begin{array}{ll}
-\Delta v+\frac{4\pi\chi(\S)}{|\S|}=2C(u)^2Ke^{v-c}&\text{in }\S\\
\de_\nu u=2C(u)he^\frac{v-c}2&\text{on }\de\S,\end{array}\right.$$
which solves \eqref{doublecurv} with the choice $c=2\log C(u)$.

Now, suppose that $u$ is a solution of \eqref{doublecurv}. Then, by \eqref{gb} one gets $\int_{\de\S}he^\frac u2=2\pi\chi(\S)-\int_\S Ke^u$ which, in the definition \eqref{c} of $C(u)$, gets $C(u)=1$ in the cases $\chi(\S)=0$ or $\chi(\S)K>0$. Therefore $u$ solves \eqref{meanfieldeq}.
\end{proof}\

\begin{remark}
Observe that $C(u)$ is one of the roots of the second order equation
$$
C^2(u) \int_\S Ke^u + C(u) \int_{\de\S} he^{\frac u2} = 2\pi \chi(\S).
$$
The choice of one root rather than the other is the reason why problems \eqref{doublecurv} and \eqref{meanfieldeq} are not totally equivalent.\\
One may verify that in \eqref{c} we chose the largest of the two roots, which is a convenient choice since, in virtue of the equivalence of both problems, we want it to be positive.
\end{remark}

\begin{remark}\label{hzero}
Let us point out that the condition $C(u)>0$ is necessary in order to guarantee the equivalence between the original problem and the mean field problem, otherwise a solution of \eqref{meanfieldeq} may be trivial. For instance, if $\chi(\S)=0$ and $\int_{\de\S}h=0$, then $u=0$ would be a solution of \eqref{meanfield} whereas the constant solution does not verify \eqref{doublecurv} unless $K\equiv0,h\equiv0$. 
\end{remark}\

Such a mean field formulation admits an energy functional which, in many cases, can be treated rather easily with respect to the direct formulation. This is one of the main reasons why this new formulation has been introduced.

\begin{proposition}\label{energy}
Solutions to \eqref{meanfieldeq} are critical points on the space $H_{\chi(\S)}$ of the following energy functional:
\beq\label{j}
\mathcal J(u):=\frac12\int_\S|\nabla u|^2-F_{\chi(\S)}\left(\int_\S Ke^u,\int_{\de\S}he^\frac u2\right);
\eeq
here,
$$H_{\chi(\S)}:=\left\{u\in\overline H^1(\S):\,\left\{\begin{array}{ll}\int_\S Ke^u>-\frac1{8\pi\chi(\S)}\left(\int_{\de\S}he^\frac u2\right)_+^2&\mbox{if }\chi(\S)>0\\\int_\S Ke^u\int_{\de\S}he^\frac u2<0&\mbox{if }\chi(\S)=0\\\int_\S Ke^u<-\frac1{8\pi|\chi(\S)|}\left(\int_{\de\S}he^\frac u2\right)_-^2&\mbox{if }\chi(\S)<0\end{array}\right.\right\};$$
$t_+:=\max\{t,0\},t_-:=\max\{-t,0\}$ denote respectively the positive and negative part of a real number $t$,
$$\overline H^1(\S):=\left\{u\in H^1(\S):\,\int_\S u=0\right\},$$
and
$$F_{\chi(\S)}(\a,\b):=\left\{\begin{array}{ll}8\pi\chi(\S)\left(\log\left(\sqrt{\b^2+8\pi\chi(\S)\a}+\b\right)+\frac\b{\sqrt{\b^2+8\pi\chi(\S)\a}+\b}\right)&\text{if }\chi(\S)>0\\-2\frac{\b^2}\a&\mbox{if }\chi(\S)=0\\8\pi|\chi(\S)|\left(-\log\left(\sqrt{\b^2-8\pi|\chi(\S)|\a}-\b\right)+\frac\b{\sqrt{\b^2-8\pi|\chi(\S)|\a}-\b}\right)&\mbox{if }\chi(\S)<0\end{array}\right..$$
\end{proposition}

\begin{proof}
One easily verifies that $F_{\chi(\S)}(\a,\b)$ is well defined and smooth provided $\a,\b$ are such that $u\in H_{\chi(\S)}$. Therefore, the critical points $u$ of $\mathcal J$ verify
$$\left\{\begin{array}{ll}
-\Delta u+\mu=\de_\a F_{\chi(\S)}\left(\int_\S Ke^u,\int_{\de\S}he^\frac u2\right)Ke^u&\text{in }\S\\
\de_\nu u=\frac12\de_\b F_{\chi(\S)}\left(\int_\S Ke^u,\int_{\de\S}he^\frac u2\right)he^\frac u2&\text{on }\de\S,\end{array}\right.$$
for some $\mu\in\R$. By the definition of $F_{\chi(\S)}(\a,\b)$ and $C(u)$ in \eqref{c} one sees that, in all three cases of $\chi(\S)$, one has
$$\de_\a F_{\chi(\S)}\left(\int_\S Ke^u,\int_{\de\S}he^\frac u2\right)=2C(u)^2,\qquad\de_\b F_{\chi(\S)}\left(\int_\S Ke^u,\int_{\de\S}he^\frac u2\right)=4C(u).$$
Finally, integrating by parts one gets $\mu=\frac{4\pi\chi(\S)}{|\S|}$, therefore a critical point $u$ solves \eqref{meanfieldeq}.
\end{proof}\

We now need to verify that the energy functional is defined in a non-empty space. This will be true under no assumption on $K,h$ other than the necessary ones to get solutions.

\begin{lemma}\label{domain}
The space $H_{\chi(\S)}$ is non-empty if and only if:
\begin{itemize}
\item[$(1)$] if $\chi(\S)>0$, $K(x)>0$ for some $x\in\S$ or $h(y)>0$ for some $x\in\de\S$;
\item[$(2)$] if $\chi(\S)=0$, $K(x)h(y)<0$ for some $x\in\S,y\in\de\S$;
\item[$(3)$] if $\chi(\S)<0$, $K(x)<0$ for some $x\in\S$ or $h(y)<0$ for some $x\in\de\S$.
\end{itemize}
Moreover, if such conditions are not satisfied, then problem \eqref{doublecurv} has no solution.
\end{lemma}

\begin{proof}
It is clear in all cases that the space is empty if the assumptions on $K,h$ are not satisfied. If this holds true when $\chi(\S)>0$, then for any solution of \eqref{doublecurv} one has $\int_\S Ke^u,\int_{\de\S} he^\frac u2$ both negative, in contradiction with the Gauss-Bonnet formula \eqref{gb}. Similarly, the assumptions on $K,h$ are necessary to get solutions in the case $\chi(\S)\le0$.\\
Let us now show the converse. In case $(1)$, if $K$ is positive somewhere, then we can choose a positive $\vfi$ supported in the region where $K>0$ and set $\d:=\inf_{\{\vfi\ge1\}}K>0$ so that
$$\mathrm{spt}(\vfi)\subset\{K>0\},\qquad\{\vfi\ge1\}\subset\{K\ge\d\},$$ 
so that for $C\gg1$ large enough one has
$$\int_\S Ke^{C\vfi}\ge\int_{\{K\ge\d\}}Ke^{C\vfi}+\int_{\{K<0\}}K\ge\d|\{K\ge\d\}|e^C-\|K\|_\infty|\S|>0.$$
Taking into account its definition, that implies the non-emptiness of $H_{\chi(\S)}$.
Similarly, if $K<0$ everywhere and $h>0$ in a portion of $\de\S$, we choose a positive $\vfi$ supported in a small neighborhood of such portion and set $\d':=\inf_{\{\vfi\ge1\}\cap\de\S}h>0$ in such a way that
$$|\mathrm{spt}(\vfi)|\le\e,\qquad\mathrm{spt}(\vfi)\cap\de\S\subset\{h>0\},\qquad\{\vfi\ge1\}\cap\de\S\subset\{h\ge\d'\},$$
for some $\e>0$. Therefore, taking $C\gg1$ large enough and then $\e\ll1$ in dependence on $C$ one has
\begin{eqnarray*}
&&\int_{\de\S}he^{\frac{C\vfi}2}-\sqrt{-8\pi\chi(\S)\int_\S Ke^{C\vfi}}\\
&\ge&\int_{\{h\ge\d\}}he^{\frac{C\vfi}2}-\int_{\{h<0\}}h-\sqrt{8\pi|\chi(\S)|\left(\int_{\mathrm{spt}(\vfi)}|K|e^\frac{C\vfi}2+\int_{\S\setminus\mathrm{spt}(\vfi)}|K|\right)}\\
&\ge&\d|\{h\ge\d\}|e^\frac C2-\|h\|_\infty|\de\S|-\sqrt{8\pi|\chi(\S)|\|K\|_\infty\e} e^\frac{C\|\vfi\|_\infty}2\\
&>&0.
\end{eqnarray*}
In case $(2)$, if $h(x)<0<K(x)$ for the same $x\in\de\S$, then one can take $\vfi$ satisfying
\begin{eqnarray*}
\mathrm{spt}(\vfi)\subset\{K>0\},&\qquad&\{\vfi\ge1\}\subset\{K\ge\d\},\\
\mathrm{spt}(\vfi)\cap\de\S\subset\{h<0\},&\qquad&\{\vfi\ge1\}\cap\de\S\subset\{h\le-\d\};
\end{eqnarray*}
hence one easily gets, for large $C$,
\begin{eqnarray*}
&&\int_\S Ke^{C\vfi}\ge\d|\{K\ge\d\}|e^C-\|K\|_\infty|\S|>0,\\
&&\int_{\de\S}he^\frac{C\vfi}2\le-\d|\{h\le-\d\}|e^\frac C2-\|h\|_\infty|\partial\S|<0;
\end{eqnarray*}
the same argument works if $K(x)<0<h(x)$ at one boundary point.\\
On the other hand, if $K,h$ have the same sign on the whole boundary, say positive, we take $\vfi_1$ supported where $K<0$ and $\vfi_2$ supported close to the boundary, as before. Since $\vfi_1|_{\de\S}\equiv0$, then $\int_{\de\S}he^\frac{C_1\vfi_1+C_2\vfi_2}2=\int_{\de\S}he^\frac{C_2\vfi_2}2>0$ for large $C_2$; on the other hand, if $C_1$ is also large enough in dependence of $C_2$, then
\begin{eqnarray*}
\int_\S Ke^{C_1\vfi_1+C_2\vfi_2}&\le&\int_{\{K\le-\d\}}Ke^{C_1\vfi_1}+\int_{\mathrm{spt}(\vfi_2)}Ke^{C_2\vfi_2}+\int_{\S\setminus(\mathrm{spt}(\vfi_1)\cup\mathrm{spt}(\vfi_2))}K\\
&\le&-\d|\{K\le-\d\}|e^{C_1}+\|K\|_\infty\left(e^{C_2\|\vfi_2\|_\infty}+|\S|\right)\\
&<&0.
\end{eqnarray*}
The case $\chi(\S)<0$ can be treated in the same way as $\chi(\S)>0$ by just changing sign to both $K$ and $h$.
\end{proof}\

\section{The case $\chi(\S)>0$}

In this section we will focus on the case $\chi(\S)>0$ to prove Theorem~\ref{thm:chipositive}.\\
In this case, problem \eqref{meanfieldeq} can be attacked variationally by means of Trudinger-Moser type inequalities. In fact, in the energy functional \eqref{j} the main term will be given by the logarithmic one (see the definition of $F_{\chi(\S)}$, since the other is uniformly bounded from above, as follows by the following elementary lemma.

\begin{lemma}\label{bounded}
Let $\chi>0$ be a fixed positive number. Then, for any $\a,\b\in\R$ such that
$$\a>-\frac1{8\pi\chi}\b_+^2$$
one has
$$\frac\b{\sqrt{\b^2+8\pi\chi\a}+\b}\le1.$$
\end{lemma}\

Trudinger-Moser inequalities are a widely-used tool for the problem of prescribing one curvature (see \cite{chyg,liliu} and \cite{moser} for closed surfaces) and the formulation introduced in the previous section allows us to extend it to the prescription of both curvatures. Such results will quickly provide an ad hoc Trudinger-Moser inequality for problem \eqref{meanfieldeq}.

\begin{proposition}\label{mt}
There exists $C>0$ such that for any $u\in\overline H^1(\S)$ there holds:
\beq\label{mtineq}
\log\left(\sqrt{\left(\int_{\de\S}e^\frac u2\right)^2+8\pi\int_\S e^u}+\int_{\de\S}e^\frac u2\right)\le\frac1{16\pi}\int_\S|\nabla u|^2+C.
\eeq
\end{proposition}

\begin{proof}
From Corollary 2.5 in \cite{chyg} there exists $C>0$ such that, for any $u\in\overline H^1(\S)$:
\beq\label{ChYgIneq}
\log\int_\S e^u\le\frac1{8\pi}\int_\S|\nabla u|^2+C;
\eeq
on the other hand, from Corollary 2.6 in \cite{sergio} we get (see also \cite{liliu}):
\beq\label{LiLiuIneq}
\log\int_{\de\S}e^{\frac u2}\le\frac1{16\pi}\int_\S|\nabla u|^2+C.
\eeq
Therefore, from the elementary inequality
\beq\label{ineq}
\sqrt{\b^2+8\pi\a}+\b\le\left(\sqrt{1+8\pi}+1\right)\max\{\sqrt\a,\b\},\qquad\forall\a,\b>0
\eeq
we deduce
\begin{eqnarray*}
&&\log\left(\sqrt{\left(\int_{\de\S}e^\frac u2\right)^2+8\pi\int_\S e^u}+\int_{\de\S}e^\frac u2\right)\\
&\le&\max\left\{\frac12\log\int_\S e^u,\log\int_{\de\S}e^{\frac u2}\right\}+\log\left(\sqrt{1+8\pi}+1\right)\\
&\le&\frac1{16\pi}\int_\S|\nabla u|^2+C+\log\left(\sqrt{1+8\pi}+1\right),
\end{eqnarray*}
which concludes the proof.
\end{proof}

\begin{corollary}
If $\chi(\S)=1$ then the energy functional $\mathcal J(u)$ is uniformly bounded from below for $u\in H_1$.
\end{corollary}

\begin{proof}
We define
\beq\label{alphabeta}
\a:=\int_\S Ke^u,\qquad\b:=\int_{\de\S}he^\frac u2.
\eeq
Using Lemma \ref{bounded}, Proposition \ref{mt} and the uniform boundedness of $h,K$ we get:
\begin{eqnarray*}
\mathcal J(u)&=&\frac12\int_\S|\nabla u|^2-8\pi\log\left(\sqrt{\b^2+8\pi\a}+\b\right)-8\pi\frac\b{\sqrt{\b^2+8\pi\chi\a}+\b}\\
&\ge&\frac12\int_\S|\nabla u|^2-8\pi\log\left(\sqrt{\left(\int_{\de\S}e^\frac u2\right)^2+8\pi\int_\S e^u}+\int_{\de\S}e^\frac u2\right)\\
&-&4\pi\log\max\left\{\|h\|_\infty^2,\|K\|_\infty\right\}-8\pi\\
&\ge&-C.
\end{eqnarray*}
\end{proof}\

Boundedness from below, however, does not ensure the existence of minimizing solutions. In fact, using standard test functions, one can easily see that the energy functional is not coercive, see Proposition 3.2 in \cite{bls}.\\
This is the reason why we assumed $K,h$ to be symmetric in Theorem \ref{thm:chipositive}: symmetry allows one to take a better constant in the inequality \eqref{mtineq} (a so-called improved Trudinger-Moser inequality) which yields coercivity, hence minimizers. The argument is similar to \cite{moser} and \cite{sergio}, Proposition 2.12.\\
To this purpose, we introduce a subset of $\overline H^1(\S)$ sharing the same $k$-symmetry as $K,h$:
$$\widetilde H_k:=\left\{u\in\overline H^1(\S):\,u(\rho_kx)=u(x)\,\text{for a.e. }x\in\S\right\},$$
where now $\S$ is the half-sphere \eqref{s} and $\rho_k$ is the rotation of $\frac{2\pi}k$ given by \eqref{rhok}.

\begin{proposition}\label{mtimproved}
For any $\e>0$, there exists $C_\e>0$ such that for any $u\in \widetilde H_k$ there holds:
\beq\label{mtimpr}
\log\left(\sqrt{\left(\int_{\de\S}e^\frac u2\right)^2+8\pi\int_\S e^u}+\int_{\de\S}e^\frac u2\right)\le\frac{1+\e}{32\pi}\int_\S|\nabla u|^2+C_\e
\eeq
\end{proposition}

\begin{proof}
We start by estimating the boundary integral.\\
We cover $\de\S$ with open balls of radius $\d\le\frac{\sqrt{2\left(1-\cos\frac{2\pi}k\right)}}2$, so that for any $x\in\de\S$ the ball of radius $2\d$ does not intersect its image under the rotation of $\frac{2\pi}k$, namely
\beq\label{distance}
B_{2\d}(x)\cap B_{2\d}(\rho_kx)=\emptyset.
\eeq
Due to compactness, one has $\de\S\subset\cup_{i=1}^NB_\d(x_i)$ for some $x_1,\dots,x_N\in\de\S$ and, up to relabeling, $\frac{\int_{B_\d(x_1)\cap\de\S}e^{\frac u2}}{\int_{\de\S}e^{\frac u2}}\ge\frac1N$. Because of symmetry, namely $u\in\widetilde H_k$, one has
$$\frac{\int_{B_\d\left(\rho_k^j x_1\right)\cap\de\S}e^{\frac u2}}{\int_{\de\S}e^{\frac u2}}=\frac{\int_{B_\d(x_1)\cap\de\S}e^{\frac u2}}{\int_{\de\S}e^{\frac u2}}\ge\frac1N,\qquad j=1,\dots,k,$$
where $\rho_k^j(x_1)=\underbrace{\rho_k\circ\dots\circ\rho_k}_j(x_1)$.
In view of \eqref{distance}, the balls $B_\d\left(\rho_k^j x_1\right)$ are at a distance greater or equal than $2\d$ with respect to each other, therefore we are in position to apply Corollary 2.11 from \cite{sergio} to get:
\beq\label{imprbdry}
\log\int_{\de\S}e^\frac u2\le\frac{1+\e}{16k\pi}\int_\S|\nabla u|^2+C_\e\le\frac{1+\e}{32\pi}\int_\S|\nabla u|^2+C_\e.
\eeq\

As for the interior integral, one of the following occurs (or both):
\begin{eqnarray}
\label{int}\text{Either }&&\frac{\int_{B_\frac12(0)}e^u}{\int_\S e^u}\ge\frac12,\\
\label{bdry}\text{or }&&\frac{\int_{\S\setminus B_\frac12(0)}e^u}{\int_\S e^u}\ge\frac12.
\end{eqnarray}
In case inequality \eqref{int} occurs, then as in Proposition 2.2 from \cite{ruizsoriano} we get
\beq\label{imprint}
\log\int_\S e^u\le\log\int_{B_\frac12(0)}e^u+\log2\le\frac{1+\e}{16\pi}\int_\S|\n u|^2+C_\e.
\eeq
Putting together \eqref{imprbdry} and \eqref{imprint} and using \eqref{ineq} as in the proof of Proposition \ref{mt} we obtain \eqref{mtimpr}.\\

On the other hand, if \eqref{bdry} occurs, then we may argue as for the boundary integral to get $\d>0,N\in\N$ and $k$ symmetric points $x_1,\rho_kx_1,\dots,\rho_k^{k-1}x_1$ such that
\begin{eqnarray*}
B_{2\d}\left(\rho_k^ix_1\right)\cap B_{2\d}\left(\rho_k^jx_1\right)=\emptyset&\qquad&i,j=1,\dots,k,\,i\ne j;\\
\frac{\int_{B_\d\left(\rho_k^jx_1\right)\cap\left(\S\setminus B_\frac12(0)\right)}e^u}{\int_\S e^u}\ge\frac12\frac{\int_{B_\d\left(\rho_k^jx_1\right)\cap\left(\S\setminus B_\frac12(0)\right)}e^u}{\int_{\S\setminus B_\frac12(0)}e^u}\ge\frac1{2N}&\qquad&i=1,\dots,k.
\end{eqnarray*}
From Corollary 2.9 in \cite{sergio} one gets
$$\log\int_\S e^u\le\frac{1+\e}{8k\pi}\int_\S|\nabla u|^2+C_\e\le\frac{1+\e}{16\pi}\int_\S|\nabla u|^2+C_\e,$$
and the conclusion follows as in the previous case.
\end{proof}\

Proposition \ref{mtimpr} gives coercivity at infinity of the energy functional. However, one needs also to take care of the fact that the energy functional will not be defined on the whole space $\overline H^1(\S)$ (or on $\widetilde H_k$) but rather on its subspace $H_1$, introduced in Proposition \ref{energy}. Therefore, in general coercivity does not suffice to get global minimizers since $\mathcal J(u)$ may not grow to $+\infty$ as $u$ approaches the boundary of the domain
$$\de H_1=\left\{u\in\overline H^1(\S):\,\int_\S Ke^u=-\frac1{8\pi}\left(\int_{\de\S}he^\frac u2\right)_+^2\right\}.$$
It is not clear, in the general case, which assumptions should be made on $K,h$ in order to prevent minimizing sequences to approach $\de H_1$. However, the assumption $K>0$ fixes this issue since it implies $\de H_1=\emptyset$.

\begin{proof}[Proof of Theorem \ref{thm:chipositive}]
Since $K\ge0,K\not\equiv0$, then $\int_\S Ke^u>0$ for any $u\in\overline H^1(\S)$, therefore $H_1=\overline H^1(\S)$.\\
If $u\in \widetilde H_k$, then by Proposition \ref{mtimproved} we get, for any $\e>0$:
\begin{eqnarray*}
\mathcal J(u)&\ge&\frac12\int_\S|\nabla u|^2-8\pi\log\left(\sqrt{\left(\int_{\de\S}e^\frac u2\right)^2+8\pi\int_\S e^u}+\int_{\de\S}e^\frac u2\right)-C\\
&\ge&\frac{1-\e}4\int_\S|\nabla u|^2-C.
\end{eqnarray*}
By choosing $\e<1$ we get $\mathcal J(u)\to+\infty$ as $\int_\S|\nabla u|^2\to+\infty$, that is $\mathcal J$ is coercive on $\widetilde H_k$. Since $\mathcal J$ is also continuous on $\widetilde H_k$, direct methods from calculus of variations give the existence of some minimizer.\\
Finally, since $\widetilde H_k$ is a natural constraint for $\mathcal J$, such a minimizer $u$ will solve \eqref{meanfieldeq}.
\end{proof}\

Similarly as \cite{sergio}, Theorem \ref{thm:chipositive} is stable under perturbation and still holds true if $K,h$ are not $k$-symmetric but still uniformly close to a $k$-symmetric function.

\begin{theorem}\label{thm:perturb}
Assume $\S$ is given by \eqref{s} and:
\begin{itemize}
\item $K(x)\ge0$ for all $x\in\S$, $K\not\equiv0$ and there exists $K_0(x),h_0(x)$ $k$-symmetric and $\e>0$ small enough such that $\|K-K_0\|_\infty+\|h-h_0\|_\infty<\e$.
\end{itemize}
Then problem \eqref{doublecurv} admits a solution.
\end{theorem}

The proof will be skipped, since it is similar to \cite{sergio}, Section 4.

\

We conclude this section by pointing out that problem \eqref{meanfieldeq} can be generalized to a general mean field problem in the same spirit as the ones studied in \cite{lin,djadli} on closed surfaces. Consider, for $\l>0$,
\beq\label{meanfieldeqlambda}
\left\{\begin{array}{ll}
-\Delta u=\l\left(\frac{\l Ke^u}{2\left(\sqrt{\left(\int_{\de\S}he^\frac u2\right)^2+\l\int_\S Ke^u}+\int_{\de\S}he^\frac u2\right)^2} - \frac1{|\S|} \right) &\text{in }\S\\
\de_\nu u=\l\frac{he^\frac u2}{\sqrt{\left(\int_{\de\S}he^\frac u2\right)^2+\l\int_\S Ke^u}+\int_{\de\S}he^\frac u2}&\text{on }\de\S,\\
\end{array}\right.
\eeq
As in Proposition \ref{energy}, solutions of \eqref{meanfieldeqlambda} are critical point of the energy functional
\beq\label{jlambda}
\mathcal J_\l(u):=\frac12\int_\S|\nabla u|^2- \l\log\left(\sqrt{\b^2+\l\a}+\b\right)-\l\frac\b{\sqrt{\b^2+\l\a}+\b},
\eeq
defined in the subdomain $H_1$ with $\a,\b$ as in \eqref{alphabeta}.\\
In the case $\chi(\S)=1$ we just conside, problems \eqref{meanfieldeqlambda} coincides with \eqref{meanfieldeq} when $\l=8\pi$, that is the limiting case when the energy functional is bounded from below but not coercive.\\
However, in the presence of conical singularities, $\chi(\S)$ may be lower or higher than $1$, corresponding to $\l<8\pi$ or $\l>8\pi$, namely the functional may be coercive or unbounded and one may find global minimizer or min-max critical points. The authors will address this problem in a subsequent paper.\\

Another problem where a mean-field formulation might be applicable is the prescription of $Q$ and $T$ curvature for 4--dimensional surfaces with boundary. The condition provided by the Gauss-Bonnet-Chern theorem, which prescribes the interior and the boundary integral terms, suggests that our strategy may offer new perspectives with respect to other approaches, see \cite{sergioazahara,changquing}.

\section{Trace inequality and the case $\chi(\S)\le0$}

In the cases $\chi(\S)=0$ and $\chi(\S)<0$ the variational approaches are rather similar to each other, but totally different from the previously studied case $\chi(\S)>0$.\\
In fact, here the logarithmic term is not present at all in the energy functional when $\chi(\S)=0$ and it is bounded from above when $\chi(\S)<0$. On the other hand, the extra term may be unbounded from above and one needs some inequalities to control such a term.\\
To this purpose, we have the following Proposition, in the same spirit as Lemma 3.2 in \cite{lsmr}.

\begin{proposition}\label{trace}
Assume $K(x)\le0$ for any $x\in\S$ and let $\mathfrak D$ be defined by $\eqref{d}$. Then, for any $\e>0$ there exists $C_\e$ such that for any $u\in H^1(\S)$ one has:
\beq\label{traceineq}
\frac{\left(\int_{\de\S}he^\frac u2\right)^2}{\int_\S|K|e^u}\le\frac{\left(\mathfrak D_M+\e\right)^2}4\int_\S|\nabla u|^2+C_\e,\qquad\text{where }\mathfrak D_M:=\max_{x\in\de\S}|\mathfrak D(x)|.
\eeq
\end{proposition}

\begin{proof}
Fix $\e>0$ and take a partition of the unit $\{\phi_j\}_j$ of $\de\S$ with $\mathrm{spt}(\phi_j)$ so small that $\frac{\max_{\mathrm{spt}(\phi_j)}|h|}{\min_{\mathrm{spt}(\phi_j)}\sqrt{|K|}}\le\mathfrak D_M+\frac\e2$, which follows from the uniform continuity. Then, take a smooth vector field $\Psi:\S\to\R^2$ such that $\Psi|_{\de\S}=\nu$ and apply the divergence theorem to $\Psi\phi_je^u$; one has
$$\int_{\de\S}|h|e^\frac u2\le C\int_\S e^\frac u2+\frac{\mathfrak D_M+\frac\e2}2\sum_j\int_\S\sqrt{|K|}\phi_je^\frac u2|\nabla u|.$$
Then, applying the Cauchy-Schwartz inequality and the elementary inequality
$$A+B\le\sqrt{(1+\e)A^2+\left(1+\frac1\e\right)B^2},$$
one gets:
\begin{eqnarray*}
\left|\int_{\de\S}he^\frac u2\right|&\le&\frac{\mathfrak D_M+\frac\e2}2\int_\S\sqrt{|K|}e^\frac u2|\nabla u|+C\int_\S\sqrt{|K|}e^\frac u2\\
&\le&\sqrt{\int_\S|K|e^u}\left(\frac{\mathfrak D_M+\frac\e2}2\sqrt{\int_\S|\nabla u|^2}+C\right)\\
&\le&\sqrt{\left(\int_\S|K|e^u\right)\left(\frac{\left(\mathfrak D_M+\e\right)^2}4\int_\S|\nabla u|^2+C_\e\right)}
\end{eqnarray*}
\end{proof}

\begin{remark}
Proposition \ref{trace} is sharp in the sense that the constant $\frac{\mathfrak D_M^2}4$ in \eqref{traceineq} cannot be improved in any case.\\
To see this, fix a small region $\G\subset\de\S$ and its tubular neighborhood $\Omega\subset\S$ in such a way that, up to changing sign to $h$, which does not change inequality \eqref{traceineq}, $\frac{h(y)}{\sqrt{|K(x)|}}\ge\mathfrak D_M-\e$ for any $x\in\Omega,y\in\G$. We may also assume that, for some $\d>0$, the diffeomorphism $\Phi:\Omega\leftrightarrow[-\d,\d]\times[0,\d]$ is $\e$-close to the identity in $C^1$-norm.\\
Then, we take
$$u_n:=v_n\circ\Phi,\qquad\mbox{with}\qquad v_n(s,t):=\eta(s,t)\xi_n(t),\qquad\xi_n(t):=-2\log\left(1+nt\right)$$
and $\eta\in C^\infty_0([-\d,\d]\times[0,\d])$ such that
$$0\le\eta\le1,\qquad\eta_{[-\d+\e,\d-\e]\times[0,\d-\e]}\equiv1,\qquad\|\nabla\eta\|_\infty\le\frac C\e;$$
therefore,
\begin{eqnarray*}
\int_{\de\S}he^\frac{u_n}2&\ge&(1-C\e)\int_{-\d}^\d\left(h\circ\Phi^{-1}\right)(s,0)e^\frac{v_n(s,0)}2\mathrm ds\\
&\ge&(1-C\e)\left(\min_\G h\right)\int_{-\d+\e}^{\d+\e}e^\frac{\xi_n(0)}2\mathrm ds\\
&=&(1-C\e)\left(\min_\G h\right)2\d;\\
&&\\
\int_\S|K|e^{u_n}&\le&(1+C\e)\int_{-\d}^\d\mathrm ds\int_0^\d\left(|K|\circ\Phi^{-1}\right)(s,t)e^{v_n(s,t)}\mathrm dt\\
&\le&(1+C\e)\max_\Omega|K|\int_{-\d}^\d\mathrm ds\int_0^\d e^{\xi_n(t)}\mathrm dt\\
&\le&(1+C\e)\max_\Omega|K|\frac{2\d}n;\\
&&\\
\int_\S|\nabla u_n|^2&\le&(1+C\e)\int_{-\d}^\d\mathrm ds\int_0^\d|\nabla v_n(s,t)|^2\mathrm dt\\
&\le&(1+C\e)\int_{-\d}^\d\mathrm ds\int_0^\d\left((1+\e)\eta_n(s,t)^2\xi_n'(t)^2+\left(1+\frac1\e\right)\xi_n(t)^2|\nabla\eta_n(s,t)|^2\right)\mathrm dt\\
&\le&(1+C\e)2\d\int_0^\d\left(\xi_n(t)^2+\frac C{\e^3}\xi_n(t)^2\right)\mathrm dt\\
&\le&(1+C\e){8\d n}+C_\e\log^2n.
\end{eqnarray*}
Therefore, if one takes $\mathfrak D_0<\mathfrak D_M$ and $\e$ small enough,
$$\frac{\left(\int_\S he^\frac{u_n}2\right)^2}{\int_\S|K|e^{u_n}}-\frac{\mathfrak D_0^2}4\int_\S|\nabla u_n|^2\ge\left((1-C\e)\left(\mathfrak D_M-\e\right)^2-\mathfrak D_0^2\right)2\d n-C_\e\log^2n\underset{n\to+\infty}\to+\infty.$$\

More surprisingly, in some cases one cannot even take $\e=0$ in inequality \eqref{traceineq}.\\
In fact, consider the case $h\equiv1,K\equiv-1$ and assume $\de\S$ has a connected component $\G$ where the geodesic curvature $h_g$ has positive average and take, similarly as before, $u_n:=v_n\circ\Phi,v_n(s,t):=\zeta(t)\xi_n(t)$ with
$$\zeta\in C^\infty([0,\d]),\qquad0\le\zeta\le1,\qquad\zeta_{\left[0,\frac\d2\right]}\equiv1.$$
Since $\det D\Phi(s,t)=1-h_g(s)t+O\left(t^2\right)$, we have
\begin{eqnarray*}
\int_{\de\S}e^\frac{u_n}2&=&|\de\S|;\\
&&\\
\int_\S e^{u_n}&\le&\int_{\de\S}\mathrm ds\int_0^\d e^{\xi_n(t)}\left(1-h_g(s)t+O\left(t^2\right)\right)\mathrm dt\\
&=&\frac{|\de\S|}n-\left(\int_{\de\S}h\right)\frac{\log n}{n^2}+O\left(\frac1{n^2}\right);\\
&&\\
\int_\S|\nabla u_n|^2&\le&\int_{\de\S}\mathrm ds\int_0^\d\left(\zeta'(t)\xi_n(t)+\zeta(t)\xi_n'(t)\right)^2\left(1-h_g(s)t+O\left(t^2\right)\right)\mathrm dt\\
&&\int_{\de\S}\mathrm ds\int_0^\d\xi_n'(t)^2\left(1-h_g(s)t+O\left(t^2\right)\right)\mathrm dt\\
&=&4|\de\S|n-4\left(\int_{\de\S}h\right)\log n+O(1).
\end{eqnarray*}
Therefore,
$$\frac{\left(\int_\S e^\frac{u_n}2\right)^2}{\int_\S e^{u_n}}-\frac14\int_\S|\nabla u_n|^2=2|\de\S|\left(\int_{\de\S}h\right)\log n+O(1)\underset{n\to+\infty}\to+\infty.$$
It would be interesting to find out under which condition on the surface $\S$ inequality \eqref{traceineq} holds true with $\e=0$ and, in the cases when it does not hold true, what type of correction terms should be added.
\end{remark}\

\begin{proof}[Proof of Theorem \ref{thm:chizero}]\
	
\begin{enumerate}
\item Since $h\le0,h\not\equiv0$, then $\int_{\de\S}he^\frac u2<0$ for any $u\in\overline H^1(\S)$, therefore
$$H_0=H_0':=\left\{u\in\overline H^1(\S):\,\int_\S Ke^u>0\right\}.$$
By definition, for any $u\in H_0'$ one has
$$\mathcal J(u)=\frac12\int_\S|\nabla u|^2+2\frac{\left(\int_{\de\S}he^\frac u2\right)^2}{\int_\S Ke^u}\ge\frac12\int_\S|\nabla u|^2\underset{\|u\|\to+\infty}\to+\infty.$$
We will now show that $\mathcal J$ diverges to $+\infty$ also when $u$ approaches the boundary of the domain
$$\de H_0'=\left\{u\in\overline H^1(\S):\,\int_\S Ke^u=0\right\};$$
this will imply $\mathcal J$ is coercive on $H_0'$, hence it has a global minimizer solving \eqref{meanfieldeq}. Assume $H_0'\ni u_n\underset{n\to+\infty}\to u_0\in\de H_0'$; by weak continuity, one has
$$\int_{\de\S}he^\frac{u_n}2\underset{n\to+\infty}\to\int_{\de\S}he^\frac{u_0}2<0,\qquad\int_\S Ke^{u_n}\underset{n\to+\infty}\to\int_\S Ke^{u_0}=0,$$
therefore
$$\mathcal J(u_n)\ge2\frac{\left(\int_{\de\S}he^\frac{u_n}2\right)^2}{\int_\S Ke^{u_n}}\underset{n\to+\infty}\to+\infty,$$
which concludes the proof.\\

\item We have $\int_\S Ke^u>0$ for any $u\in\overline H^1(\S)$, therefore $\mathcal J(u)$ is defined for any $u\in\overline H^1(\S)$. Arguing as in case $(1)$, $\mathcal J$ is coercive on $\overline H^1(\S)$, hence it has a minimizer which solves
\beq\label{falsemeanfield}
\left\{\begin{array}{ll}
-\Delta u=2\frac{\left(\int_{\de\S}he^\frac u2\right)^2}{\left(\int_\S Ke^u\right)^2}Ke^u&\text{in }\S\\
\de_\nu u=-2\frac{\int_{\de\S}he^\frac u2}{\int_\S Ke^u}he^\frac u2&\text{on }\de\S.\\
\end{array}\right..
\eeq
However, solutions to \eqref{falsemeanfield} may not solve \eqref{meanfieldeq}, because $C(u)=-\frac{\int_{\de\S}he^\frac u2}{\int_\S Ke^u}$ may be zero or negative, both of which we want to rule out.\\
If $C(u)=0$, then the only solution to \eqref{falsemeanfield} is $u\equiv0$, which would give
$$\int_{\de\S}h=-C(u)\int_\S K=0,$$
in contradiction with the assumptions (see also Remark \ref{hzero}).
Suppose now $C(u)<0$; since \eqref{falsemeanfield} is invariant under changing sign to $h$ and $C(u)$ is odd with respect to $h$, then $u$ solves
$$\left\{\begin{array}{ll}
-\Delta u=2C(u)^2Ke^u&\text{in }\S\\
\de_\nu u=2C(u)(-h)e^\frac u2&\text{on }\de\S,\\
C(u)=-\frac{\int_{\de\S}(-h)e^\frac u2}{\int_\S Ke^u}>0
\end{array}\right..$$
Therefore, due to Proposition \ref{meanfield}, problem \eqref{doublecurv} with $-h$ instead of $h$ has a solution:
$$\left\{\begin{array}{ll}
-\Delta u=2Ke^u&\text{in }\S\\
\de_\nu u=-2he^\frac u2&\text{on }\de\S,
\end{array}\right..$$
However, multiplying the equation by $e^{-\frac u2}$ and integrating by parts gets
$$0>-\int_\S|\nabla u|^2e^{-\frac u2}=\int_\S Ke^\frac u2-\int_{\de\S}h>-\int_{\de\S}h,$$
again contradicting the assumptions. This shows that the minimizer on $\overline H^1(\S)$ actually solves \eqref{meanfieldeq}.\\

\item We have $\int_\S Ke^u<0$ and $\int_{\de\S}he^\frac u2>0$ for any $u\in\overline H^1(\S)$, therefore $H_0=\overline H^1(\S)$, hence in order to get minimizing solutions, we suffice to show coercivity at infinity. To this purpose, we take a small $\e$ such that $\mathfrak D(y)+\e<1$ for any $y\in\de\S$ and apply Proposition \ref{trace}:
$$\mathcal J(u)=\frac12\int_\S|\nabla u|^2-2\frac{\left(\int_{\de\S}he^\frac u2\right)^2}{\int_\S|K|e^u}\ge\underbrace{\frac{1-\left(\mathfrak D_M+\e\right)^2}2}_{>0}\int_\S|\nabla u|^2-C_\e\underset{\|u\|\to+\infty}\to+\infty.$$
Then we find a minimizing solution by applying the direct method.
\end{enumerate}
\end{proof}\

Theorem \ref{thm:chizero} can be extended by dropping the assumptions on the sign of $h$, but if $h$ is allowed to change sign we do not know if problem \eqref{doublecurv} can be solved for $h$ or $-h$.

\begin{theorem}\label{thm:hminush}
Assume $\chi(\S)=0$ and:
\begin{itemize}
\item $K(x)\le0$ for all $x\in\S$, $K\not\equiv0$, $\int_{\de\S}h\ne0$ and $|\mathfrak D(y)|<1$ for all $y\in\de\S$.
\end{itemize}
Then, problem \eqref{doublecurv} admits a solution for $h$ or $-h$.
\end{theorem}

As in Theorem \ref{thm:chizero}, this is a sub-case of Theorem 1.2 in \cite{lsmr}, which gives existence of solutions under assuming $\int_{\de\S}h>0$.

\begin{proof} As in case $(3)$ in Theorem \ref{thm:chizero}, $\mathcal J$ is well-defined and coercive on $\overline H^1(\S)$ and, as in case $(2)$, it has a minimizers solving \eqref{falsemeanfield}. $C(u)=0$ is excluded since it would give $\int_\S h=0$, so one may have either $C(u)>0$ or $C(u)<0$: in the former case, \eqref{doublecurv} has a solution for $h$; in the latter case, \eqref{doublecurv} has a solution for or $-h$.
\end{proof}\

In the case $\chi(\S)<0$ the arguments are similar as the case $\chi(\S)=0$.\\
In fact, even though the nonlinear terms $F_{\chi(\S)}(\a,\b)$ is different than $F_0(\a,\b)$, it has essentially the same asymptotic behavior for positive $\b$, whereas for negative $\b$ it is bounded from below. To be more precise, we have the following elementary lemma. 

\begin{lemma}\label{lemmaft}
For any $\e>0$ there exists $C_\e>0$ such that the function
\beq\label{ft}
f(t):=F_{\chi(\S)}(-1,t)=8\pi|\chi(\S)|\left(-\log\left(\sqrt{t^2+8\pi|\chi(\S)|}-t\right)+\frac t{\sqrt{t^2+8\pi|\chi(\S)|}-t}\right)
\eeq
verifies
$$f(t)\le(2+\e)t_+^2+C_\e.$$
\end{lemma}

In view of lemma \ref{lemmaft}, the case $\chi(\S)<0$ can still be treated using Proposition \ref{trace}.

\begin{proof}[Proof of Theorem \ref{thm:chinegative}]
From assuming $K\le0,K\not\equiv0$ we get $\int_\S Ke^u<0$ for any $u\in\overline H^1(\S)$, hence $H_{\chi(\S)}=\overline H^1(\S)$.\\
We may write the energy functional as
$$\mathcal J(u)=\frac12\int_\S|\nabla u|^2+4\pi|\chi(\S)|\log(-\a)+4\pi|\chi(\S)|\log(8\pi)-f\left(\frac\b{\sqrt{-\a}}\right),$$
where $\a,\b$ are as in \eqref{alphabeta} and $f$ is as in \eqref{ft}.\\
The first nonlinear term is uniformly bounded from below on $\overline H^1(\S)$ because of the Jensen's inequality:
$$\int_\S(-K)e^u=\int_\S e^{u+\log|K|}\ge|\S|e^{\frac1{|\S|}\int_\S(u+\log|K|)}=|\S|e^{\frac1{|\S|}\int_\S\log|K|}.$$
By assumption, we have
$$\mathfrak D_M^+:=\max_{x\in\de\S}\frac{h(x)}{\sqrt{|K(x)|}}=\max_{x\in\de\S}\frac{h_+(x)}{\sqrt{|K(x)|}}<1;$$
therefore, using Lemma \ref{lemmaft} and Proposition \ref{trace} with $h_+$ instead of $h$ we get, for $\e$ small enough:
\begin{eqnarray*}
\mathcal J(u)&\ge&\frac12\int_\S|\nabla u|^2+f\left(\frac\b{\sqrt{-\a}}\right)-C\\
&\ge&\frac12\int_\S|\nabla u|^2-(2+\e)\frac{\b_+^2}{-\a}-C_\e\\
&\ge&\frac12\int_\S|\nabla u|^2-(2+\e)\frac{\left(\int_{\de\S}h_+e^\frac u2\right)^2}{\int_\S|K|e^u}-C_\e\\
&\ge&\underbrace{\frac{2-(2+\e)\left(\mathfrak D_M^++\e\right)^2}4}_{>0}\int_\S|\nabla u|^2-C_\e\\
&\underset{\|u\|\to+\infty}\to&+\infty,
\end{eqnarray*}
hence $\mathcal J$ has minimizers which solve \eqref{meanfieldeq}.
\end{proof}

\

{\bf Acknowledgements:} The first author is partially supported by the group INdAM-GNAMPA project ``Propriet\`a qualitative delle soluzioni di equazioni ellittiche'' CUP\_E53C22001930001.\\ The second author is currently supported by the grant Juan de la Cierva Incorporación fellowship (JC2020-046123-I), funded by MCIN/AEI/10.13039/501100011033, and by the European Union Next Generation EU/PRTR. He is also partially supported by Grant PID2021-122122NB-I00 funded by MCIN/AEI/ 10.13039/501100011033 and by “ERDF A way of making Europe”. \\
The work was partially carried out during the first author's visit to University of Granada, to which he is grateful.\\
The authors wish to thank David Ruiz for many fruitful discussions and suggestions. They also express their gratitude for the suggestions and comments made by the anonymous referee.

\bibliographystyle{plain}
\bibliography{bls}

\end{document}